\documentclass[a4paper,oneside,12pt]{amsart} 

\usepackage[ansinew]{inputenc} 
\usepackage[T1]{fontenc}
\usepackage[english]{babel} 

\usepackage{enumerate} 
\usepackage{booktabs}

\usepackage{amsmath} 
\usepackage{amsfonts} 
\usepackage{amssymb}
\usepackage{amsthm}
\usepackage{latexsym}
\usepackage{bbm}
\usepackage{setspace}
\usepackage[top=3cm , bottom=3cm , left=2.8cm , right=2.8cm]{geometry}
\usepackage[pdftex]{graphicx}
  \DeclareGraphicsExtensions{.pdf,.png,.jpg,.jpeg,.mps}
  \usepackage{pgf}
  \usepackage{tikz}
\usepackage{floatflt}
\usepackage{url}

\makeatletter

\@namedef{subjclassname@2010}{%

  \textup{2010} Mathematics Subject Classification}

\makeatother

\theoremstyle{definition} 

\newtheorem{defn}{Definition}
\newtheorem*{notation}{Notation}

\theoremstyle{plain}
\newtheorem{prop}[defn]{Proposition}
\newtheorem{lemma}[defn]{Lemma}
\newtheorem{theorem}[defn]{Theorem}

\theoremstyle{remark}

\newtheorem*{acks}{Acknowledgements}

\title[Local Hardy spaces]{A note on local Hardy spaces}
\author{Mikko Kemppainen}
\address{Department of Mathematics and Statistics, University of Helsinki, FI-00014 Helsinki, Finland}
\email{mikko.k.kemppainen@helsinki.fi}
\date{\today}

\begin{document}

\begin{abstract}
  We consider a non-negative self-adjoint operator $L$ satisfying
  generalized Gaussian estimates on a doubling metric measure space,
  and show that if $L$ has a spectral gap then the local and global
  Hardy spaces defined by means of appropriate square functions coincide.
\end{abstract}



\maketitle

\section{Introduction}

The real-variable Hardy space $H^1$ was introduced by
E. M. Stein and G. Weiss \cite{STEINHp,STEINWEISSHp} 
as a harmonic analytic substitute for the endpoint 
Lebesgue space $\mathcal{L}^1$. A larger local Hardy space $\mathfrak{h}^1$,
which is better suited for smooth Fourier multipliers, was
studied by D. Goldberg \cite{GOLDBERG}.

The classical setting is based on the Euclidean space $\mathbb{R}^n$ and 
arises, somewhat implicitly, from its Laplacian 
$\Delta = \partial^2/\partial x_1^2 + \cdots + \partial^2/\partial x_n^2$. 
The theory has since then been extended into various directions: 
to second order elliptic operators
\cite{HOFMANNDIV,SECONDORDER} and Schr\"odinger operators \cite{DZIU}
on the Euclidean space;
to the first order framework of Hodge--Dirac operators on 
Riemannian manifolds with doubling volume property \cite{AMR}
and to the corresponding local setting \cite{MORRIS};
to non-negative self-adjoint operators satisfying Davies--Gaffney
estimates of order $2$ on doubling metric measure spaces
\cite{HOFMANNHARDY}
and to operators with estimates of higher order \cite{KUNSTMANNUHL,UHLTHESIS,FREYTHESIS}.
Local Hardy spaces for operators with pointwise Gaussian upper bounds
have also been considered in \cite{DACHUNLOCAL,LIXINLOCAL}.

The main result of this article concerns non-negative self-adjoint
operators $L$ that have a
spectral gap in the sense that $\inf \sigma(L) =: \lambda_0 > 0$. 
In such a case, the size of any function is captured by its 
\emph{local} square function, which only takes into 
account short time diffusion of the semigroup $(e^{-tL})_{t>0}$.
Indeed, denoting by $E$ the spectral measure of $L$, we have
\begin{equation*}
  \frac{1}{4} \| f \|_2^2 = \Big\| \Big( \int_0^1 |tLe^{-tL}f|^2 \,
  \frac{dt}{t} \Big)^{1/2} \Big\|_2^2 + \int_{\lambda_0}^\infty
  \Big( \frac{\lambda}{2} + \frac{1}{4}\Big) e^{-2\lambda} \,
  dE_{f,f}(\lambda) ,
\end{equation*}
and therefore
\begin{equation*}
  \Big\| \Big( \int_0^1 |tLe^{-tL}f|^2 \,
  \frac{dt}{t} \Big)^{1/2} \Big\|_2 \geq c \| f \|_2 ,
\end{equation*}
where $c = (\frac{1}{4} - ( \frac{\lambda_0}{2} + \frac{1}{4} ) e^{-2\lambda_0})^{1/2} > 0$.
In other words, the global square function is controlled by its local
version.

Assume that $(M,d,\mu)$ is a doubling metric measure space:
there exists a number $n>0$ such that for every ball $B\subset M$ we have
\begin{equation*}
  \mu (\alpha B) \lesssim \alpha^n \mu (B) ,
\end{equation*}
whenever $\alpha \geq 1$.
Assume moreover that $L$ is a non-negative self-adjoint 
operator of `order $m\geq 2$' on the Lebesgue space $\mathcal{L}^2$ (see
Definition \ref{genGaussian}).

The Hardy space $H_L^1$ is defined by means of the conical square
function
\begin{equation*}
  Sf(x) = \Big( \int_0^\infty \frac{1}{\mu (B(x,t))}
  \int_{B(x,t)} |t^mL e^{-t^mL}f(y)|^2 \, d\mu (y) \, \frac{dt}{t}
  \Big)^{1/2} , \quad x\in M,
\end{equation*}
as a completion with respect to the norm 
$\| f \|_{H_L^1} = \| Sf \|_1$. Similarly, the local Hardy space
$\mathfrak{h}_L^1$ is defined by means of $S_{loc}$ in which $\int_0^\infty$
is replaced by $\int_0^1$ and the norm is defined as
$\| f \|_{\mathfrak{h}_L^1} = \| S_{loc}f \|_1 + \| f \|_1$.
From $\| Sf \|_1 \gtrsim \| f \|_1$ it is immediate that
$H_L^1 \subset \mathfrak{h}_L^1$. For the purposes of this article we may
think of these as incomplete subspaces of $\mathcal{L}^2$ in order
to avoid technical complications.

The main result (Theorem \ref{mainthm}) says that if 
$(e^{-tL})_{t>0}$ satisfies generalized Gaussian $(1,2)$-estimates
(see Definition \ref{genGaussian}) and
$L$ has a spectral gap, then actually $\mathfrak{h}_L^1 = H_L^1$. 

A prototypical example of a second order 
operator to which the result applies is the Schr\"odinger operator
$L=-\Delta + |x|^2$ on $\mathbb{R}^n$.
Moreover, by \cite[Theorem 3.1]{MAXMINGEN}, the heat kernel of
any generalized Schr\"odinger operator $L=(-\Delta)^{m/2} + V$ 
with non-negative locally integrable potential $V$ on $\mathbb{R}^n$
satisfies pointwise Gaussian upper bounds of order $m$
(and therefore also generalized Gaussian $(1,2)$-estimates)
whenever $m$ is an even integer greater than $n$.

\begin{notation}
  We denote by $\textsf{D} (L)$ and $\textsf{R} (L)$ the domain and range in 
  $\mathcal{L}^2$ of the linear operator $L$. 
  The notation $\| \cdot \|_{p\to 2}$ stands for the operator norm
  from $\mathcal{L}^p$ to $\mathcal{L}^2$. 
  The radius of a ball $B\subset M$ is denoted by $r_B$.
\end{notation}

\begin{acks}
The author is a member of the Finnish Centre of Excellence in Analysis and Dynamics Research.
\end{acks}

\section{Off-diagonal estimates}

In this section we discuss the generalized Gaussian estimates, which
in this form and generality were introduced by S. Blunck and P. C. 
Kunstmann in 
\cite{BLUNCKKUNSTMAX,BLUNCKKUNST,BLUNCKKUNSTLEGENDRE}.

\begin{defn}
\label{genGaussian}
  Let $m\geq 2$ and $p\in [1,2]$. 
  A family $(T_t)_{t>0}$ of linear operators on $\mathcal{L}^2$ is
  said to satisfy \emph{generalized Gaussian $(p,2)$-estimates 
  (of order $m$)} (abbreviated $\textup{GGE}_m(p,2)$) 
  if there exists a constant $c$ such that 
  for all $x,y\in M$ and all $t>0$ we have
  \begin{equation*}
    \| 1_{B(x,t)} T_{t^m} 1_{B(y,t)} \|_{p\to 2}
    \lesssim \mu (B(x,t))^{-(\frac{1}{p} - \frac{1}{2})}
    \exp \Big( -c \Big( \frac{d(x,y)}{t} \Big)^{\frac{m}{m-1}} \Big) .
  \end{equation*}
\end{defn}

For a ball $B\subset M$ we write $C_k(B) = 2^kB\setminus 2^{k-1}B$,
when $k\geq 1$, $C_0(B)= B$,
and $C_k^*(B) = 2^{k+1} B\setminus 2^{k-2}B$, 
$C_1^*(B) = 4B$, $C_0^*(B) = 2B$. 

The following result collects the required off-diagonal estimates.
A more systematic treatment can be found for instance in
M. Uhl's thesis \cite{UHLTHESIS}.

\begin{prop}
\label{offdiag}
  If $(T_t)_{t>0}$ satisfies $\textup{GGE}_m(p,2)$ with
  $1\leq p \leq 2$, then for every ball $B\subset M$ and every $k\geq 1$ 
  we have
  \begin{equation*}  
    \| 1_{C_k(B)} T_{t^m} 1_B \|_{p\to 2}
    \lesssim \mu (B)^{-n(\frac{1}{p} - \frac{1}{2})} 
      \Big( 1+\frac{r_B}{t} \Big)^{n(\frac{1}{p} - \frac{1}{2})} 2^{nk}
    \exp \Big( -c \Big( \frac{2^kr_B}{t} \Big)^{\frac{m}{m-1}} \Big) ,
    \quad t>0. 
  \end{equation*}
  Moreover, if $(T_t)_{t>0}$ satisfies $\textup{GGE}_m(2,2)$, then
  for every ball $B\subset M$ and every $k\geq 0$ we have
   \begin{equation*}  
    \| 1_{C_k(B)} T_{t^m} 1_{M\setminus C_k^*(B)} \|_{2\to 2}
    \lesssim \Big( \frac{t}{2^kr_B} \Big)^{n+2}, \quad t>0.
  \end{equation*}
\end{prop}
\begin{proof}
  Assume that $(T_t)_{t>0}$ satisfies $\textup{GGE}_m(p,2)$ with
  $1\leq p \leq 2$ and consider a ball $B\subset M$.
  By \cite[Lemma 2.6 b)]{UHLTHESIS} we have for all $j\geq 2$ and $t>0$,
  \begin{equation}
  \label{uhl}
    \| 1_{(j+1)B\setminus jB} T_{t^m} 1_B \|_{p\to 2}
    \lesssim \mu (B)^{-n(\frac{1}{p} - \frac{1}{2})} 
      \Big( 1+\frac{r_B}{t} \Big)^{n(\frac{1}{p} - \frac{1}{2})} j^n \exp \Big( -c \Big( \frac{jr_B}{t} \Big)^{\frac{m}{m-1}} \Big) .
  \end{equation}
  Let $k\geq 1$. Writing
  \begin{equation*}
    C_k(B) = \bigcup_{j=2^{k-1}}^{2^k-1} (j+1)B\setminus jB ,
  \end{equation*}
  and noticing that
  \begin{equation*}
    \sum_{j=2^{k-1}}^{2^k-1} j^n \lesssim 2^{nk} ,
  \end{equation*}
  we obtain from \eqref{uhl} that for all $t>0$,
  \begin{equation}
  \label{firstclaim}
    \| 1_{C_k(B)} T_{t^m} 1_B \|_{p\to 2}
    \lesssim \mu (B)^{-n(\frac{1}{p} - \frac{1}{2})} 
      \Big( 1+\frac{r_B}{t} \Big)^{n(\frac{1}{p} - \frac{1}{2})}
    2^{nk} \exp \Big( -c \Big( \frac{2^kr_B}{t} \Big)^{\frac{m}{m-1}} \Big) ,
  \end{equation}
  which proves the first case.
  
  For the second case we assume that $(T_t)_{t>0}$ satisfies 
  $\textup{GGE}_m(2,2)$. 
  Writing, for any ball $B\subset M$,
  \begin{equation*}
    M\setminus 2B = \bigcup_{j=2}^\infty (j+1)B\setminus jB 
  \end{equation*}
  and making use of the estimate
  \begin{equation}
  \label{decayest}
    \exp (-c\alpha^{\frac{m}{m-1}}) 
    \lesssim \alpha^{-n-2} \exp (-c'\alpha^{\frac{m}{m-1}})
  \end{equation}
  we see that for all $t>0$
  \begin{align*}
    \sum_{j=2}^\infty j^n \exp \Big( -c \Big( \frac{jr_B}{t} \Big)^{\frac{m}{m-1}} \Big)
    &\leq \sum_{j=2}^\infty \frac{t^n}{r_B^n} 
    \Big( \frac{jr_B}{t} \Big)^n \exp \Big( -c \Big( \frac{jr_B}{t} \Big)^{\frac{m}{m-1}} \Big) \\
    &\lesssim \sum_{j=2}^\infty \Big( \frac{t}{r_B} \Big)^{n+2} 
    j^{-2} \\
    &\lesssim \Big( \frac{t}{r_B} \Big)^{n+2} .
  \end{align*}
  From \eqref{uhl} we now obtain by self-adjointness that
  for every ball $B\subset M$,
  \begin{equation}
  \label{C_0}
    \| 1_B T_{t^m} 1_{M\setminus 2B} \|_{2\to 2}
    = \| 1_{M\setminus 2B} T_{t^m} 1_B \|_{2\to 2}
    \lesssim \Big( \frac{t}{r_B} \Big)^{n+2} ,
    \quad t>0.
  \end{equation}
  Moreover, writing $1_{C_1(B)} = 1_{2B} - 1_B$, we see that
  \begin{equation*}
  \begin{split}
     \| 1_{C_1(B)} T_{t^m} 1_{M\setminus C_1^*(B)} \|_{2\to 2}
     &=  \| 1_{2B\setminus B} T_{t^m} 1_{M\setminus 4B} \|_{2\to 2} \\
    &\leq \| 1_{2B} T_{t^m} 1_{M\setminus 4B} \|_{2\to 2}
    + \| 1_B T_{t^m} 1_{M\setminus 4B} \|_{2\to 2} \\
    &\lesssim \Big( \frac{t}{2r_B} \Big)^{n+2}
    \end{split}
  \end{equation*}
  Finally, for any ball $B\subset M$ and any $k\geq 2$, we have
  \begin{equation*}
    1_{C_k(B)} = 1_{2^kB} - 1_{2^{k-1}B} , \quad 
    1_{M\setminus C_k^*(B)} = 1_{M\setminus 2^{k+1}B} + 1_{2^{k-2}B}.
  \end{equation*}
  Therefore, by applying \eqref{C_0} with balls $2^kB$ and $2^{k-1}B$,
  and \eqref{firstclaim} with the ball $2^{k-2}B$, we obtain
  \begin{align*}
    \| 1_{C_k(B)} T_{t^m} 1_{M\setminus C_k^*(B)} \|_{2\to 2}
    &\leq \| 1_{2^kB} T_{t^m} 1_{M\setminus 2(2^kB)} \|_{2\to 2}
    + \| 1_{2^{k-1}B} T_{t^m} 1_{M\setminus 2(2^kB)} \|_{2\to 2} \\
    &+ \| 1_{C_2(2^{k-2}B)} T_{t^m} 1_{2^{k-2}B} \|_{2\to 2} \\
    &\lesssim \Big( \frac{t}{2^kr_B}\Big)^{n+2}
    + 2^{2n}
    \exp \Big( -c\Big( \frac{2^{k-2}r_B}{t} \Big)^{\frac{m}{m-1}} \Big) ,
    \quad t>0,
  \end{align*}
  where by \eqref{decayest} we have
  \begin{equation*}
    \exp \Big( -c\Big( \frac{2^{k-2}r_B}{t} \Big)^{\frac{m}{m-1}} \Big)
    \lesssim \Big( \frac{t}{2^kr_B}\Big)^{n+2} ,
  \end{equation*}
  as required.
\end{proof}

The main result of this article relies on the following 
fact that a spectral gap implies exponential decay in time.
Recall \cite[Lemma 2.9]{UHLTHESIS}: if $(e^{-tL})_{t>0}$
satisfies $\textup{GGE}_m(2,2)$, then for any integer $j\geq 0$, 
also $((tL)^j e^{-tL})_{t>0}$ satisfies $\textup{GGE}_m(2,2)$.

\begin{prop}
\label{expdecay}
  If $(e^{-tL})_{t>0}$ satisfies $\textup{GGE}_m(2,2)$ and
  $\inf \sigma (L) > 0$, then, for small $\delta > 0$, also
  $(e^{-t(L-\delta)})_{t>0}$ satisfies $\textup{GGE}_m(2,2)$.
  In this case, for some $\delta > 0$ we have
  \begin{equation*}
    \| 1_{E'} t^mLe^{-t^mL} 1_E \|_{2\to 2} \lesssim e^{-\delta t^m},
    \quad t>0,
  \end{equation*}
  whenever $E,E'\subset M$. Moreover, for every ball $B\subset M$ 
  and every $k\geq 0$ we have 
  \begin{equation*}
    \| 1_{C_k(B)} t^mL e^{-t^mL} 1_{M\setminus C_k^*(B)} \|_{2\to 2}
    \lesssim e^{-\delta t^m} \Big( \frac{t}{2^kr_B} \Big)^{n+2},
    \quad t>0 .
  \end{equation*}
\end{prop}
\begin{proof}
  The first claim follows by a straightforward 
  generalization of \cite[Proposition 2.2]{SIKORACOUL}:
  \emph{If $F$ is a bounded analytic function on $\mathbb{C}_+$ and
  $|F(t)| \lesssim \exp (\delta t -\gamma t^{-\frac{1}{m-1}})$
  for all $t>0$, then
  $|F(t)| \lesssim \exp (-\gamma t^{-\frac{1}{m-1}})$ for all $t>0$.}
  The same proof applies with the modification 
  $u(\zeta) = F((\gamma/\zeta)^{m-1})$. This can then be applied to
  $F(z) = \langle e^{-z(L-\delta)} f, g \rangle$ and
  $\gamma = d(E,E')^{\frac{m}{m-1}}$ for arbitrary
  $f,g\in\mathcal{L}^2$ whenever $0 < \delta < \inf \sigma (L)$.
  
  For such $\delta$ we then have
  \begin{equation*}
    \| 1_{E'} t^mL e^{-t^m(L-\delta)} 1_E \|_{2\to 2}
    \leq \| 1_{E'} t^m(L-\delta) e^{-t^m(L-\delta)} 1_E \|_{2\to 2} 
    + \delta t^m \| 1_{E'} e^{-t^m(L-\delta)} 1_E \|_{2\to 2},
  \end{equation*}
  for all $t>0$ whenever $E,E'\subset M$, and therefore
  \begin{equation*}
    \| 1_{E'} t^mLe^{-t^mL} 1_E \|_{2\to 2} \lesssim 
    (1 + \delta t^m) e^{-\delta t^m} \lesssim e^{-\delta t^m / 2},
    \quad t>0,
  \end{equation*}
  by uniform boundedness. Moreover, Proposition \ref{offdiag} implies 
  that for every ball $B\subset M$ and every $k\geq 0$ we have
  \begin{equation*}
    \| 1_{C_k(B)} t^mL e^{-t^mL} 1_{M\setminus C_k^*(B)} \|_{2\to 2}
    \lesssim e^{-\delta t^m /2}
    \Big( \frac{t}{2^kr_B} \Big)^{n+2} , \quad t>0,
  \end{equation*}
  as required.
\end{proof}

\section{Molecular decomposition}

In this section we establish a molecular decomposition
on the local Hardy space $\mathfrak{h}_L^1$.
For comparison, see \cite[Subsection 4.3]{UHLTHESIS},
\cite[Chapter 5]{HOFMANNHARDY}, \cite[Subsection 7.1]{MORRIS}
and \cite[Theorem 1.3]{LIXINLOCAL}.

\begin{theorem}
\label{moleculardec}
  Assume that $(e^{-tL})_{t>0}$ satisfies $\textup{GGE}_m(1,2)$.
  Then every $f\in \mathfrak{h}^1_L \cap \textsf{R} (L)$ admits a decomposition into 
  $N$-molecules $a_j$ for any integer $N\geq 1$ in the sense that
  \begin{equation*}
    f = \sum_j \lambda_j a_j , \quad \textup{where} \quad
    \sum_j |\lambda_j| \eqsim \| f \|_{\mathfrak{h}^1_L} .
  \end{equation*}
  Here the series converges in $\mathcal{L}^2$.
\end{theorem}

\begin{defn}
  Let $N\geq 1$ be an integer.
  A function $a\in \mathcal{L}^2$ is called an \emph{$N$-molecule} 
  associated with a
  ball $B\subset M$ if
  \begin{equation*}
    \| 1_{C_k(B)} a \|_2 \leq 2^{-k} \mu(2^kB)^{-1/2}, \quad k\geq 0,
  \end{equation*}
  and either $r_B \geq 1$, or $r_B\leq 2$ and  
  $a$ is cancellative in the following sense:
  there exists a $b\in \textsf{D} (L^N)$ such that $a = L^Nb$ and
  for all $j=0,1,\ldots , N$ we have
  \begin{equation*}
    \| 1_{C_k(B)} (r^m_BL)^j b \|_2 \leq r^{mN}_B 2^{-k} \mu (2^kB)^{-1/2}, \quad k\geq 0.
  \end{equation*}
\end{defn}

\begin{lemma}
  For any integer $N\geq 1$ there exist constants 
  $c_0,c_1,\ldots , c_{N+2}$
  such that
  \begin{equation*}
    f = c_{N+2} \int_0^1 (t^mL)^{N+2} e^{-2t^mL} f \, \frac{dt}{t}
    + c_{N+1}L^{N+1}e^{-2L}f + \cdots + c_1 Le^{-2L} f + c_0 e^{-2L}f
  \end{equation*}
  for any $f\in \textsf{R} (L)$.
\end{lemma}
\begin{proof}
  Fixing an integer $N\geq 1$ and an $f\in\textsf{R} (L)$
  we start from the usual Calder\'on reproducing formula
  \begin{equation*}
    f = c_{N+2} \Big( \int_0^1 + \int_1^\infty \Big) (t^mL)^{N+2} e^{-2t^mL} f \, \frac{dt}{t}
  \end{equation*}
  and show that
  \begin{equation}
  \label{intformula}
    c_{N+2} \int_1^\infty (t^mL)^{N+2} e^{-2t^mL} f \, \frac{dt}{t}
    = c_{N+1}L^{N+1}e^{-2L}f + \cdots + c_1 Le^{-2L} f + c_0 e^{-2L}f
  \end{equation}
  for some constants $c_0,c_1,\ldots ,c_{N+1}$.
  Now, integrating by parts we see that for any $j\geq 1$ and for
  all $\lambda \geq 0$ we have
  \begin{align*}
    \int_1^\infty (t^m\lambda)^{j+1} e^{-2t^m\lambda} \, \frac{dt}{t}
    &= -\frac{1}{2} \int_1^\infty (t^m\lambda)^j 
    \partial_t(e^{-2t^m\lambda}) \, dt \\
    &= -\frac{1}{2} \Big[ (t^m\lambda)^j e^{-2t^m\lambda} \Big]_{t=1}^\infty + \frac{1}{2} \int_1^\infty \partial_t ((t^m\lambda)^j) e^{-2t^m\lambda} \, dt \\
    &= -\frac{1}{2} \lambda^j e^{-2\lambda} + \frac{mj}{2}
    \int_1^\infty (t^m\lambda)^j e^{-2t^m\lambda} \, \frac{dt}{t} .
  \end{align*}
  Iterating this from $j=N+1$ to $j=0$ and using the spectral theorem
  we obtain \eqref{intformula}.
\end{proof}

\begin{proof}[Proof of Theorem \ref{moleculardec}]
  Let $f\in\mathfrak{h}_L^1 \cap \textsf{R} (L)$ and
  let $N\geq 1$ be an integer. We rewrite the reproducing formula as
  \begin{equation*}
    f = \pi_1u + \pi_2f , \quad u(\cdot , t) = t^mL e^{-t^mL}f ,
  \end{equation*}
  where
  \begin{equation*}
    \pi_1u = c_{N+2} \int_0^1 (t^mL)^{N+1} e^{-t^mL} u(\cdot , t) \,
    \frac{dt}{t}
  \end{equation*}
  and
  \begin{equation*}
    \pi_2f = c_{N+1}L^{N+1}e^{-2L}f + \cdots + c_1Le^{-2L}f 
    + c_0e^{-2L}f .
  \end{equation*}
  By atomic decomposition on local tent spaces 
  (see \cite[Theorem 3.6]{MORRIS} or \cite[Theorem 4.5]{NONNEGATIVE})
  we can write
  \begin{equation*}
    u = \sum_j \lambda_j u_j
  \end{equation*}
  where $u_j$ are \emph{tent atoms} in the sense that each is supported in
  a box $B \times (0,r_B)$ and has
  \begin{equation*}
    \Big( \int_0^{r_B} \| u(\cdot , t) \|_2^2 \, \frac{dt}{t} \Big)^{1/2} 
    \leq \mu (B)^{-1/2} . 
  \end{equation*}

  Claim: If $u$ is a tent atom in $B\times (0,r_B)$,
  then $\pi_1u$ is a constant multiple of a cancellative 
  $N$-molecule associated with $B$.
  Choosing
  \begin{equation*}
    b = \int_0^1 t^{m(N+1)} Le^{-t^mL}u(\cdot , t) \, \frac{dt}{t}
  \end{equation*}
  we have $L^Nb = \pi_1u$. Let $j=0,1,\ldots , N$.
  For any $g$ supported in $C_k(B)$, $k\geq 0$, 
  with $\| g \|_2 \leq 1$ we have
  \begin{align*}
    |\langle (r_B^mL)^jb , g \rangle |
    &= \Big| \int_M (r_B^mL)^j \int_0^1 t^{m(N+1)} Le^{-t^mL}
    u(\cdot , t) \, \frac{dt}{t} \, g \, d\mu \Big|  \\
    &= \Big| \int_0^{r_B\wedge 1} \int_B u(\cdot , t) (r_B^mL)^j t^{m(N+1)} Le^{-t^mL} g \, d\mu \, \frac{dt}{t} \Big|  \\
    &\leq \Big( \int_0^{r_B} \| u(\cdot , t) \|_2^2 \, \frac{dt}{t} \Big)^{1/2}
    \Big( \int_0^{r_B\wedge 1} \| 1_B (r_B^mL)^j t^{m(N+1)}Le^{-t^mL}g
    \|_2^2 \, \frac{dt}{t} \Big)^{1/2} \\
    &\leq \mu (B)^{-1/2} r_B^{mj}
    \Big( \int_0^{r_B\wedge 1} t^{2m(N-j)} \| 1_B (t^mL)^{j+1} e^{-t^mL}g
    \|_2^2 \, \frac{dt}{t} \Big)^{1/2} .
  \end{align*}
  For $k=0,1$ we simply estimate
  \begin{equation*}
    \| 1_B (t^mL)^{j+1} e^{-t^mL}g \|_2
    \lesssim \| g \|_2 \leq 1,
  \end{equation*}
  while for $k\geq 2$ we see that, since $t \leq r_B$,
  \begin{equation*}
    \| 1_B (t^mL)^{j+1} e^{-t^mL}g \|_2
    \lesssim 2^{nk} 
    \exp \Big( -c\Big( \frac{2^kr_B}{t} \Big)^{\frac{m}{m-1}} \Big)
    \| g \|_2
    \lesssim 2^{-\frac{n}{2}k} \exp (-c'2^k) .
  \end{equation*}
  Therefore for all $k\geq 0$ we have
  \begin{align*}
  \| 1_{C_k(B)} (r^m_BL)^j b \|_2 
  &\lesssim \mu (B)^{-1/2} 2^{-\frac{n}{2}k} \exp (-c'2^k)
  r^{mj}_B \Big( \int_0^{r_B} t^{2m(N-j) - 1}\, dt \Big)^{1/2} \\
  &\lesssim r^{mN}_B 2^{-k} \mu (2^kB)^{-1/2} ,
  \end{align*}
  and the claim has been verified.

  Turning to $\pi_2$ we cover $M$ with a countable 
  disjoint family $\mathcal{Q}$ of sets $Q$ each of which is contained
  in a ball $B_Q$ of radius one. Then we may write
  \begin{equation*}
    \pi_2f = \sum_{Q\in\mathcal{Q}} \Big( 
    c_{N+1}L^{N+1}e^{-2L}(1_Qf) + \cdots + c_1Le^{-2L}(1_Qf) 
    + c_0e^{-2L}(1_Qf) \Big) ,
  \end{equation*}
  where each $c_jL^j e^{-2L}(1_Qf)$ is a constant multiple of a 
  noncancellative atom associated with the ball $B_Q$: 
  since $(tLe^{-tL})_{t>0}$ satisfies $\textup{GGE}_m(1,2)$,
  it follows from Proposition \ref{offdiag} that 
  for every $k\geq 0$,
  \begin{equation*}
    \| 1_{C_k(B_Q)} L^j e^{-2L}(1_Qf) \|_2
    \lesssim \mu (B)^{-1/2} 2^{nk} \exp (-c2^{\frac{m}{m-1}k}) \| f \|_1
    \lesssim 2^{-k} \mu (2^kB)^{-1/2} .
  \end{equation*}
  This finishes the proof.
\end{proof}

\section{Main result}

\begin{theorem}
\label{mainthm}
  Assume that $(e^{-tL})_{t>0}$ satisfies $\textup{GGE}_m(1,2)$.
  If $\inf \sigma (L) > 0$, then 
  $\| Sf \|_1 \lesssim \| S_{loc}f \|_1 + \| f \|_1$, i.e.
  $\mathfrak{h}_L^1 = H_L^1$.
\end{theorem}
\begin{proof}
  The proof of the required estimate 
  $\| Sf \|_1 \lesssim \| S_{loc}f \|_1 + \| f \|_1$
  reduces to showing that $\| S_\infty a \|_1 \lesssim 1$
  for all $N$-molecules $a$ when $N$ is large enough (here
  $S_\infty$ is defined as $S$, but with $\int_1^\infty$ replacing
  $\int_0^\infty$).
  Indeed, an arbitrary $f\in\mathfrak{h}_L^1 \cap \mathcal{L}^2$ can be
  decomposed into $N$-molecules $a_j$ for any $N\geq 1$ so that
  $f = \sum_j \lambda_j a_j$, and therefore
  \begin{equation*}
    \| Sf \|_1 \leq \| S_{loc}f \|_1 + \sum_j |\lambda_j | 
    \| S_\infty a_j \|_1 .
  \end{equation*}
  
  Let $a$ be an $N$-molecule associated with a ball $B\subset M$.
  If $a$ is cancellative and $N > n/(2m)$, then 
  $\| S_\infty a \|_1 \leq \| Sa \|_1 \lesssim 1$ by 
  \cite[Lemma 4.13]{UHLTHESIS} 
  (see also \cite[Corollary 3.6]{KUNSTMANNUHL}).
  We may therefore assume that $a$ is noncancellative and $r_B\geq 1$.
  Consider the following decomposition (cf. \cite[Subsection 6.2]{AMR}):
  \begin{equation*}
    M \times (1,\infty) = \Big( \bigcup_{k=0}^\infty 
    C_k(B) \times (1, 2^kr_B) \Big) \cup
    \Big( \bigcup_{k=1}^\infty 
    2^{k-1}B \times (2^{k-1}r_B, 2^kr_B) \Big) .
  \end{equation*}
  Then, by H\"older's inequality,
  \begin{equation*}
    \begin{split}
    \| S_\infty a \|_1 &\lesssim \sum_{k=0}^\infty
    \mu (2^kB)^{1/2} \Big( \int_1^{2^kr_B} \| 1_{C_k(B)} t^mLe^{-t^mL} a
    \|_2^2 \, \frac{dt}{t} \Big)^{1/2} \\
    &+ \sum_{k=1}^\infty
    \mu (2^kB)^{1/2} \Big( \int_{2^{k-1}r_B}^{2^kr_B} \| 1_{2^{k-1}B} t^mLe^{-t^mL} a
    \|_2^2 \, \frac{dt}{t} \Big)^{1/2},
    \end{split}
  \end{equation*}
  and it suffices to show that,
  \begin{equation}
  \label{mainestimate1}
  \Big( \int_1^\infty \| 1_{C_k(B)} t^mLe^{-t^mL} a \|_2^2 \, \frac{dt}{t} \Big)^{1/2}
  \lesssim 2^{-k} \mu (2^kB)^{-1/2} , \quad k\geq 0,
  \end{equation}
  and (since $r_B \geq 1$)
  \begin{equation}
  \label{mainestimate2}
  \Big( \int_{2^{k-1}}^\infty \| 1_{2^{k-1}B} t^mLe^{-t^mL} a
    \|_2^2 \, \frac{dt}{t} \Big)^{1/2}
  \lesssim 2^{-k} \mu (2^kB)^{-1/2} , \quad k\geq 1.
  \end{equation}
  
  We prove \eqref{mainestimate1} for a fixed $k\geq 0$ 
  by studying $a$ in two 
  pieces $1_{C_k^*(B)}a$ and $1_{M\setminus C_k^*(B)}a$.
  Firstly, by Proposition \ref{expdecay}, for all $t>0$ we have
  \begin{equation*}
    \| 1_{C_k(B)} t^mLe^{-t^mL} (1_{C_k^*(B)}a) \|_2
    \lesssim e^{-\delta t} \| 1_{C_k^*(B)}a \|_2 ,
  \end{equation*}
  and therefore
  \begin{align*}
    \Big( \int_1^\infty \| 1_{C_k(B)} t^mLe^{-t^mL} (1_{C_k^*(B)}a)
    \|_2^2 \, \frac{dt}{t} \Big)^{1/2}
    &\lesssim \| 1_{C_k^*(B)}a \|_2
    \Big( \int_1^\infty e^{-2\delta t} \, \frac{dt}{t} \Big)^{1/2} \\
    &\lesssim 2^{-k} \mu (2^kB)^{-1/2} .
  \end{align*}
  Secondly, by Proposition \ref{expdecay}, for all $t>0$ we have
  \begin{equation*}
    \| 1_{C_k(B)} t^mLe^{-t^mL} (1_{M\setminus C_k^*(B)}a) \|_2
    \lesssim e^{-\delta t} \Big( \frac{t}{2^kr_B} \Big)^{n+2} \| a \|_2 .
  \end{equation*}
  Therefore
  \begin{align*}
    \Big( \int_1^\infty \| 1_{C_k(B)} t^mLe^{-t^mL} (1_{M\setminus C_k^*(B)}a)
    \|_2^2 \, \frac{dt}{t} \Big)^{1/2}
    &\lesssim (2^kr_B)^{-n-2} \| a \|_2
    \Big( \int_1^\infty e^{-2\delta t} t^{2n+3} \, dt \Big)^{1/2} \\
    &\lesssim 2^{-(n+2)k} \mu (B)^{-1/2} \\
    &\lesssim 2^{-k} \mu (2^kB)^{-1/2}
  \end{align*}
  where in the second step we used the assumption that $r_B \geq 1$.
  
  We then turn to \eqref{mainestimate2} and, fixing a $k\geq 1$, 
  divide $a$ into
  $1_{2^kB}a$ and $1_{M\setminus 2^kB}a$. Again, by Proposition
  \ref{expdecay}, for all $t>0$ we have
  \begin{equation*}
    \| 1_{2^{k-1}B} t^mLe^{-t^mL} (1_{2^kB} a) \|_2
    \lesssim e^{-\delta t} \| 1_{2^kB}a \|_2 ,
  \end{equation*}
  where
  \begin{equation*}
  \begin{split}
    \| 1_{2^kB}a \|_2 \leq \sum_{l=0}^k \| 1_{C_l(B)}a \|_2
    &\leq \sum_{l=0}^k 2^{-l}\mu (2^lB)^{-1/2}
    \lesssim 2^{nk/2} \mu(2^kB)^{-1/2} .
  \end{split}
  \end{equation*}
  Therefore
  \begin{equation*}
  \begin{split}
  &\Big( \int_{2^{k-1}}^\infty \| 1_{2^{k-1}B}
    t^mL e^{-t^mL}(1_{2^kB} a) \|_2^2 \, \frac{dt}{t} \Big)^{1/2} \\
    &\lesssim 2^{(n/2 + 1)k} \Big( \int_{2^{k-1}}^\infty e^{-2\delta t}
    \, \frac{dt}{t} \Big)^{1/2} 2^{-k} \mu(2^kB)^{-1/2} ,
  \end{split}
  \end{equation*}
  where
  \begin{equation*}
    \Big( \int_{2^{k-1}}^\infty e^{-2\delta t}
    \, \frac{dt}{t} \Big)^{1/2}
    \lesssim \Big( \int_{2^{k-1}}^\infty \frac{dt}{t^{n+3}} \Big)^{1/2}
    \eqsim 2^{-k(n/2+1)} ,
  \end{equation*}
  as required.
  Finally, noting that $2^{k-1}B = C_0(2^{k-1}B)$ and 
  $M\setminus 2^kB = M\setminus C_0^*(2^{k-1}B)$, we apply 
  Proposition \ref{expdecay} to see that for all $t>0$ we have
  \begin{equation*}
    \| 1_{2^{k-1}B} t^mLe^{-t^mL} (1_{M\setminus 2^kB} a) \|_2
    \lesssim e^{-\delta t} \Big( \frac{t}{2^kr_B} \Big)^{n+2} \| a \|_2
  \end{equation*}
  and therefore
  \begin{equation*}
  \begin{split}
    &\Big( \int_{2^{k-1}}^\infty \| 1_{2^{k-1}B}
    t^mL e^{-t^mL}(1_{M\setminus 2^kB} a) \|_2^2 \, \frac{dt}{t} \Big)^{1/2}\\
    &\lesssim (2^kr_B)^{-n-2} \| a \|_2
    \Big( \int_1^\infty e^{-2\delta t} t^{2n+3} \, dt \Big)^{1/2} \\
    &\lesssim 2^{-k} \mu (2^kB)^{-1/2},
  \end{split}
  \end{equation*}
  which concludes the proof.

\end{proof}

\bibliographystyle{plain}

\begin{thebibliography}{10}

\bibitem{NONNEGATIVE}
A.~Amenta and M.~Kemppainen.
\newblock Non-uniformly local tent spaces.
\newblock {\em Publ. Mat.}, 59(1):245--270, 2015.

\bibitem{AMR}
P.~Auscher, A.~McIntosh, and E.~Russ.
\newblock Hardy spaces of differential forms on {R}iemannian manifolds.
\newblock {\em J. Geom. Anal.}, 18(1):192--248, 2008.

\bibitem{BLUNCKKUNSTMAX}
S.~Blunck and P.~C. Kunstmann.
\newblock Weighted norm estimates and maximal regularity.
\newblock {\em Adv. Differential Equations}, 7(12):1513--1532, 2002.

\bibitem{BLUNCKKUNSTLEGENDRE}
S.~Blunck and P.~C. Kunstmann.
\newblock Generalized {G}aussian estimates and the {L}egendre transform.
\newblock {\em J. Operator Theory}, 53(2):351--365, 2005.

\bibitem{BLUNCKKUNST}
S.~Blunck and P.~C. Kunstmann.
\newblock Calder\'on-{Z}ygmund theory for non-integral operators and the
  {$H^\infty$} functional calculus.
\newblock {\em Rev. Mat. Iberoamericana}, 19(3):919--942, 2003.

\bibitem{MORRIS}
A. Carbonaro, A. McIntosh, and A.~J. Morris.
\newblock Local {H}ardy spaces of differential forms on {R}iemannian manifolds.
\newblock {\em J. Geom. Anal.}, 23(1):106--169, 2013.

\bibitem{SIKORACOUL}
T. Coulhon and A. Sikora.
\newblock Gaussian heat kernel upper bounds via the {P}hragm\'en-{L}indel\"of
  theorem.
\newblock {\em Proc. Lond. Math. Soc. (3)}, 96(2):507--544, 2008.

\bibitem{MAXMINGEN}
Q. Deng, Y. Ding, and X. Yao.
\newblock Maximal and minimal forms for generalized {S}chr\"odinger operators.
\newblock {\em Indiana Univ. Math. J.}, 63(3):727--738, 2014.

\bibitem{DZIU}
J. Dziuba{\'n}ski and J. Zienkiewicz.
\newblock Hardy space {$H^1$} associated to {S}chr\"odinger operator with
  potential satisfying reverse {H}\"older inequality.
\newblock {\em Rev. Mat. Iberoamericana}, 15(2):279--296, 1999.

\bibitem{FREYTHESIS}
D. Frey.
\newblock {\em Paraproducts via $H^\infty$-functional calculus and a
  $T(1)$-{T}heorem for non-integral operators}.
\newblock PhD thesis, Karlsruher Institut f\"ur Technologie (KIT), 2011.

\bibitem{GOLDBERG}
D. Goldberg.
\newblock A local version of real {H}ardy spaces.
\newblock {\em Duke Math. J.}, 46(1):27--42, 1979.

\bibitem{LIXINLOCAL}
R. M. Gong, J.~Li, and L.~Yan.
\newblock A local version of {H}ardy spaces associated with operators on metric
  spaces.
\newblock {\em Sci. China Math.}, 56(2):315--330, 2013.

\bibitem{HOFMANNHARDY}
S.~Hofmann, G.~Lu, D.~Mitrea, M.~Mitrea, and L.~Yan.
\newblock Hardy spaces associated to non-negative self-adjoint operators
  satisfying {D}avies-{G}affney estimates.
\newblock {\em Mem. Amer. Math. Soc.}, 214(1007):vi+78, 2011.

\bibitem{HOFMANNDIV}
S.~Hofmann and S.~Mayboroda.
\newblock Hardy and {BMO} spaces associated to divergence form elliptic
  operators.
\newblock {\em Math. Ann.}, 344(1):37--116, 2009.

\bibitem{SECONDORDER}
S.~Hofmann, S.~Mayboroda, and A.~McIntosh.
\newblock Second order elliptic operators with complex bounded measurable
  coefficients in {$L^p$}, {S}obolev and {H}ardy spaces.
\newblock {\em Ann. Sci. \'Ec. Norm. Sup\'er. (4)}, 44(5):723--800, 2011.

\bibitem{DACHUNLOCAL}
R. Jiang, D. Yang, and Y. Zhou.
\newblock Localized {H}ardy spaces associated with operators.
\newblock {\em Appl. Anal.}, 88(9):1409--1427, 2009.

\bibitem{KUNSTMANNUHL}
P.~C. Kunstmann and M. Uhl.
\newblock Spectral multiplier theorems of {H}\"ormander type on {H}ardy and
  {L}ebesgue spaces.
\newblock {\em J. Operator Theory}, 73(1):27--69, 2015.

\bibitem{STEINHp}
E. M. Stein.
\newblock Classes {$H^{p}$}, multiplicateurs et fonctions de
  {L}ittlewood-{P}aley.
\newblock {\em C. R. Acad. Sci. Paris S\'er. A-B}, 263:A716--A719, 1966.

\bibitem{STEINWEISSHp}
E.~M. Stein and G. Weiss.
\newblock On the theory of harmonic functions of several variables. {I}. {T}he
  theory of {$H^{p}$}-spaces.
\newblock {\em Acta Math.}, 103:25--62, 1960.

\bibitem{UHLTHESIS}
M. Uhl.
\newblock {\em Spectral multiplier theorems of {H}\"ormander type via
  generalized {G}aussian estimates}.
\newblock PhD thesis, Karlsruher Institut f\"ur Technologie (KIT), 2011.

\end{thebibliography}
\def\cprime{$'$} \def\cprime{$'$}

\end{document}